\documentclass[12pt]{amsart}

\usepackage{amsmath,amsfonts,amssymb,amsthm}
\usepackage{graphics,graphicx,color}

\newtheorem{theorem}{Theorem}[section]
\newtheorem{lemma}[theorem]{Lemma}
\newtheorem{proposition}[theorem]{Proposition}
\newtheorem{corollary}[theorem]{Corollary}
\newtheorem{definition}[theorem]{Definition}

\def \Rm {\mathbb R}

\def \A {\mathcal{A}}
\def \tA {\tilde{\mathcal{A}}}
\def \T {\mathcal{T}}

\newcommand{\eps}{\varepsilon}

\newcommand{\norm}{\|}

\newcommand{\tE}{\tilde{E}}
\newcommand{\tEE}{\tilde{E}_1}
\newcommand{\tF}{\tilde{F}}
\newcommand{\ta}{\tilde{a}}
\newcommand{\tk}{\tilde{k}}
\newcommand{\tiA}{\tilde{A}}
\newcommand{\tbeta}{\tilde{\beta}}
\newcommand{\talpha}{\tilde{\alpha}}
\newcommand{\tgamma}{\tilde{\gamma}}

\begin{document}

\title{Stability of the Gauge Equivalent Classes in Stationary Inverse Transport}

\author{Stephen McDowall}
\thanks{First author partly supported by  NSF Grant No.~0553223}
\address{\hskip-\parindent
Stephen McDowall\\Department of Mathematics\\
Western Washington University\\
516 High Street\\ Bellingham, WA 98225-9063}
\email{stephen.mcdowall@wwu.edu}
\author{Plamen Stefanov}
\thanks{Second author partly supported by  NSF Grant No.~0554065}
\address{\hskip-\parindent
Plamen Stefanov\\Department of Mathematics\\
Purdue University\\
150 N. University Street\\
West Lafayette, IN 47907-2067} \email{stefanov@math.purdue.edu}
\author{Alexandru Tamasan}
\address{\hskip-\parindent
Alexandru Tamasan\\
Department of Mathematics\\
University of Central Florida\\
4000 Central Florida Blvd.\\Orlando, FL, 32816, USA}
\email{tamasan@math.ucf.edu}

\begin{abstract}
For anisotropic attenuating media, the albedo operator determines
the scattering and the attenuating coefficients up to a gauge
transformation. We show that such a determination is stable.
\end{abstract}

\maketitle

\pagestyle{myheadings}
\markboth{S. McDowall, P. Stefanov and A. Tamasan}{Stability of
the gauge equivalent classes}


\section{Introduction}
\label{sec:intro}
%
This paper concerns the problem of recovering the absorption and
scattering properties of a bounded, convex medium
$\Omega\subset\Rm^n$, $n\geq 2$ from the spatial-angular
measurements of the density of particles at the boundary
$\partial\Omega$. Provided that the particles interact with the
medium but not with each other, the radiation transfer in the
steady-state can be modeled  by the transport equation
\begin{equation}\label{transport_eq}
-\theta\cdot\nabla
u(x,\theta)-a(x,\theta)u(x,\theta)+\int_{S^{n-1}}k(x,\theta',\theta)u(x,\theta')d\theta'=0,
\end{equation}
for $x\in\Omega$ and $\theta\in S^{n-1}$; see, e.g.
\cite{caseZweifel67, RS}. The function $u(x,\theta)$ represents
the density of particles at $x$ traveling in the direction
$\theta$, $a(x,\theta)$ is the attenuation coefficient at $x$ for
particles moving in the direction of $\theta$, and
$k(x,\theta',\theta)$ is the scattering coefficient (or the
collision kernel) which accounts for particles from an arbitrary
direction $\theta'$ which scatter in the direction of travel
$\theta$. Let $\Gamma_\pm$ denote the incoming and outgoing
``boundary"
\begin{align}\label{Gamma_Omega}\Gamma_\pm := \{(x,\theta)\in
\partial\Omega\times S^{n-1}:\quad \pm \theta\cdot n(x)>0 \},\end{align}
$n(x)$ being the outer unit normal at a boundary point
$x\in\partial\Omega$. The medium is probed with the given
radiation
\begin{equation}\label{bd_condition_in}
u|_{\Gamma_-}=f_-.
\end{equation}
The exiting radiation $ u|_{\Gamma_+} $ is detected thus defining
the albedo operator $\A$ that takes the incoming flux $f_-$ to the
outgoing flux $u|_{\Gamma_+}$, i.e. $\A[f_-]:=u|_{\Gamma_+}$.

In general, the boundary value problem \eqref{transport_eq} and
\eqref{bd_condition_in} may not be uniquely solvable but it has a
unique  solution under some physically relevant subcritical
conditions like \eqref{critic_CS}, \eqref{critic_DL}, or
\eqref{subcritical_2D}. We note, however, that for sufficiently
regular coefficients, the problem has unique solution for generic
$(a,k)$, see \cite{stefanovUhlmann08,S-T09}.

One of the inverse boundary value problems in transport theory is
to recover the attenuation coefficient $a$ and the scattering
kernel $k$ from knowledge of the albedo operator $\A$. This
problem has been solved under some restrictive assumptions (e.g.
$k$ of a special type or independent of a variable) in
\cite{anikonov75,anikonov84, anikonovBubnov88,bal00, larsen84,
Larsen, mcCormick92,  tamasan02, tamasan03}. In three or higher
dimensions, uniqueness and reconstruction results for general $k$
and $a=a(x)$ were established in \cite{choulliStefanov99}. The
general approach is based on the study of the singularities of the
fundamental solution of \eqref{transport_eq} (see also \cite{Bo}),
and the singularities of the Schwartz kernel of $\A$. Stability
estimates for $k$ of special type   were established in
\cite{R,jnWang99}; and recently, for general $k$, in
\cite{balJollivet08}. Uniqueness and reconstruction results in a
Riemannian geometry setting, including recovery of a simple
metric, were established in \cite{macdowall04}. Similar results
for the time-dependent model were established in \cite{CS1}, and
in \cite{KLU} for the Riemannian case. In planar domains the work
in \cite{stefanovUhlmann03} shows stable determination of the
isotropic absorption and small scattering, and an extension to
simple Riemannian geometry is given in \cite{macdowall05}. Also in
two dimensional domains we point out that the recovery of $k$ is
only known under smallness conditions which are more restrictive
than what is needed to solve the direct problem; e.g. more
restrictive than \eqref{critic_CS} or \eqref{critic_DL} below. On
the other hand, in the time-dependent case, the extra variable
allows us to treat the planar case without such restrictions, see
\cite{CS1}. We also mention here the recent works
\cite{balLangmoreMonard08, L, LM}, in which the coefficients are
recovered  from angularly averaged measurements rather than from
the knowledge of the whole albedo operator $\A$. For an exhaustive
account on the inverse transport problem we refer to the review
paper \cite{bal09}.

The above mentioned results concern media with directionally independent
absorption $a=a(x)$, except for transport with variable speed when the
attenuation may depend on $|v|$, $a=a(x,|v|)$.

The attenuation accounts not only for the absorption of particles,
but also for the loss of particles due to the scattering. In the
physical case in which $k$ depends on two independent directions,
the attenuation is inherently anisotropic $a=a(x,\theta)$.
However, in an anisotropic attenuating media, the unique
determination of the coefficients from boundary measurements no
longer holds: In \cite{S-T09} it is shown that the albedo operator
determines the pairs of coefficients up to a {\em gauge
transformation}; see \eqref{s01} below. This non-uniqueness
motivates the following definition.

\begin{definition}\label{albedo_equivalence_def}
Two pairs of coefficients $(a,k)$ and $(\ta,\tk)$ are called {\em
gauge equivalent} if there exists a positive $\phi\in
L^\infty(\Omega\times S^{n-1})$ with $\theta\cdot\nabla_x
\phi(x,\theta)\in L^\infty(\Omega\times S^{n-1})$ and $\phi= 1$
for $x\in
\partial\Omega$ such that
\begin{equation}\label{s01}
\ta(x,\theta)=a(x,\theta)-\theta\cdot\nabla_x \log \phi(x,\theta),
\quad \tk(x,\theta',\theta) =
\frac{\phi(x,\theta)}{\phi(x,\theta')}k(x,\theta',\theta).
\end{equation}
We denote the equivalence class of $(a,k)$ by $\langle a,k\rangle$, and the equivalence relation itself by $\sim$.
\end{definition}
The relation defined above is reflexive since $(a,k)\sim (a,k)$
via $\phi\equiv 1$; it is symmetric since $(a,k)\sim
(\tilde{a},\tilde{k})$ via $\phi$ yields
$(\tilde{a},\tilde{k})\sim (a,k)$ via $1/\phi$; and it is
transitive since if $(a,k)\sim (\tilde{a},\tilde{k})$ via $\phi$
and $(\tilde{a},\tilde{k})\sim (a',k')$ via $\tilde{\phi}$ then
$(a,k)\sim (a',k')$ via $\phi\tilde{\phi}$.

The main result in \cite{S-T09} is that, in dimensions $n\geq 3$,
$\A=\tA$ if and only if $(a,k)\sim(\ta,\tk)$, i.e., uniqueness up
to gauge transformations. The uniqueness up to the gauge
transformation extends naturally to refractive media and to
dimension two, see \cite{MST09}.

In this paper we study the question of stability of the
determination of the gauge equivalent classes. Let $(M,\|\cdot\|_M)$
and $(N,\|\cdot\|_N)$ be Banach spaces in which the attenuation and,
respectively, the scattering kernel are considered,
$(a,k),(\ta,\tk)\in M\times N$. The distance $\Delta$ between
equivalence classes with respect to $M\times N$ is given by the
infimum of the distances between all possible pairs of representatives. More precisely,
\begin{align}
\Delta(\langle a,k\rangle,\langle \ta,\tk\rangle):=\inf_{
(a',k')\in\langle a,k\rangle,
(\ta',\tk')\in\langle\ta,\tk\rangle}\max\{\|a'-\ta'\|_M,
\|k'-\tk'\|_N\}.
\end{align}

For $n\geq 3$ we work within the class of coefficients
\begin{align}\label{continuous_coeff}
(a,k)\in L^\infty(\Omega\times S^{n-1})\times
L^\infty(\Omega\times S^{n-1};L^1( S^{n-1})).
\end{align}For two dimensional domains ($n=2$) both coefficients are assumed
bounded:
\begin{align}\label{bounded_coeff_2d}
(a,k)\in L^\infty(\Omega\times S^{1})\times L^\infty(\Omega\times
S^{1}\times S^{1}).
\end{align}The following norms are used throughout
\begin{align*}
\|a\|_\infty&= \text{ess sup}_{(x,\theta)\in\Omega\times
S^{n-1}}|a(x,\theta)|,\\
\|k\|_{\infty,1}&=\text{ess sup}_{(x,\theta')\in\Omega\times
S^{n-1}}\int_{S^{n-1}}|k(x,\theta',\theta)|d\theta,\\
\|k\|_\infty&=\text{ess sup}_{(x,\theta',\theta)\in\Omega\times
S^{1}\times S^{1}}\left|k(x,\theta',\theta)\right|,\\
\|k\|_1&=\int_\Omega\int_{S^{n-1}}\int_{S^{n-1}}|k(x,\theta',\theta)|dxd\theta'd\theta.
\end{align*}

Note that the gauge transformations preserve the class of
coefficients in \eqref{continuous_coeff} or
\eqref{bounded_coeff_2d}.

In Section \ref{albedo_Extension} we reduce the original inverse problem in
$\Omega$ to the inverse problem of transport in a larger (strictly convex)
domain $B_R\supset\overline\Omega$, where the attenuation and scattering
coefficients are extended by zero in $B_R\setminus\Omega$. More precisely
we show that the difference of two albedo operators realizes an isometry
when transported from $\partial\Omega$ to $\partial B_R$. For simplicity,
the larger domain is a ball but this is not essential.  Let $(\ta,\tk)$ be
another pair of admissible coefficients for which the forward problem in
$\Omega$ is well posed and let $\tA$ denote the corresponding albedo
operators. Set $a=\ta=0$ and $k=\tk=0$ in $B_R\setminus\overline\Omega$.
Then the forward problems in $B_R$ are also well posed and let $\A^R$ and
$\tA^R $, denote the corresponding albedo operators respectively. The
boundary data is considered on
\begin{align}\label{Gamma_B_R}\Gamma_{\pm}^R :=
\{(x,\theta)\in
\partial B_R\times S^{n-1}:\quad \pm \theta\cdot n(x)>0 \},\end{align}
$n(x)$ now being the outer unit normal at a boundary point
$x\in\partial B_R$. Provided that the forward problem is well-posed
in $L^p$, $1\leq p\leq \infty$, we show that
\begin{align}\label{isometry}
\|\A-\tA\|_{\mathcal{L}(L^p(\Gamma_-;d\xi);L^p(\Gamma_+;d\xi))}=
\|\A^R-\tA^R\|_{\mathcal{L}(L^p(\Gamma^R_-;d\xi^R);L^p(\Gamma^R_+;d\xi^R))}.
\end{align}In \eqref{isometry} we used $d\xi=|n(x)\cdot\theta|d\mu(x)d\theta$, where $d\theta$ is
the normalized measure on the sphere, $d\mu(x)$ is the induced
Lebesgue measure on $\partial\Omega$ and $n(x)$ is the unit outer
normal at some $x\in\partial\Omega$.  Similarly,
$d\xi^R=|n(x)\cdot\theta|d\mu^R(x)d\theta$, where $d\mu^R(x)$ is
the induced Lebesgue measure on $\partial B_R$.

Consequently, we may consider the data (albedo operators) given
directly on the $\partial B_R$ and drop $R$ from their notation. The
isometry \eqref{isometry} with $p=1$ is used for domains in three or
higher dimensions. For brevity let
\[\|\A-\tA\|:=\|\A-\tA\|_{\mathcal{L}(L^1(\Gamma_-;d\xi);L^1(\Gamma_+;d\xi))}.\]

Let $\tau_\pm(x,\theta)$ be the travel time it takes a particle at
$x\in B_R$ to reach the boundary $\partial B_R$ while moving in the
direction of $\pm\theta$ and define
$\tau(x,\theta)=\tau_-(x,\theta)+\tau_+(x,\theta)$. Since we work
with unit-speed velocities, $\tau(x,\theta)\leq 2R$. Moreover, since
$\mbox{dist}(\overline\Omega,\partial B_R)>0$, we have
\begin{align}\label{nontangential}
c_R:=\inf\{\tau(x,\theta):~(x,\theta)\in \overline\Omega\times
S^{n-1}\}>0.
\end{align}
Note that we could make $c_R=1$ at the expense of a sufficiently
large radius $R$.

For domains in three or higher dimensions, and for $\Sigma,\rho>0$
we consider the class
\begin{align}
U_{\Sigma,\rho}:=\{(a,k)~~\mbox {as in}~
\eqref{bounded_coeff_2d}:~\|a\|_\infty\leq\Sigma,~\|k\|_{\infty,1}\leq\rho\}.
\end{align}
The main result of stability of gauge equivalent classes is the
following.
\begin{theorem}\label{main_thm}
Let $(a,k),(\ta,\tk)\in U_{\Sigma,\rho}$ be such that the
corresponding forward problems are well posed. Then
\begin{equation}
\Delta(\langle a,k\rangle,\langle \ta,\tk\rangle)\leq
C\norm\A-\tA\norm,
\end{equation}
where $\Delta$ is with respect to $L^\infty\times  L^1$, and  $C$ is
a constant depending only on $\Sigma$, $\rho$, $c_R$ and $R$. More
precisely, there exists a representative $(a',k')\in\langle
a,k\rangle$ such that
\begin{align}
\|a'-\ta\|_\infty&\leq C\norm\A-\tA\norm,\label{closeness_in_a}\\
\|k'-\tk\|_1&\leq C\norm\A-\tA\norm,\label{closeness_in_k}
\end{align}where
\begin{align}\label{constant}
C=\max\{\pi Re^{2R\Sigma}\left(1+2\rho e^{4R\Sigma}\right),{
e^{4R\Sigma}}/{c_R}\}.
\end{align}
\end{theorem}

For the stability of the equivalence classes in two dimensional
domains we need a more refined notion of distance between the
albedo operators: Following \cite[Proposition
1]{stefanovUhlmann03}, the Schwartz kernel of a albedo operator
$\A$ admits the singular decomposition:
\begin{align}\label{2d_singular_decomposition}
\alpha=\frac{A(x',\theta')}{n(x)\cdot\theta}
\delta_{\{x'+\tau_+(x',\theta')\theta'\}}(x)\delta_{\{\theta'\}}(\theta)+\beta(x,\theta,x',\theta'),
\end{align}
where
\begin{align}\label{strentgh_of_singularity}
A(x',\theta')=\exp\left(-\int_0^{\tau_+(x',\theta')}a(x'+t\theta',\theta')dt\right)
\end{align}and $|\theta\times\theta'|\beta\in
L^\infty(\Gamma^R_+\times\Gamma^R_-)$; see also Section
\ref{preliminaries_2D}.

Let $\tilde{A},\tilde{\beta}$ be the coefficients corresponding to
the decomposition \eqref{2d_singular_decomposition} of another
albedo operator $\tA$. We define
\begin{align}\label{star_norm}
\|\A\|_*=\max\{\|A\|_\infty;\|\,|\theta\times\theta'|\beta\|_\infty\}.
\end{align}
In Section \ref{albedo_Extension} we show that the isometric
transportation from $\partial\Omega$ to $\partial B_R$ holds in
the new operator norm, i.e.,
\begin{align}\label{isometry*}
\|\A-\tA\|_*= \|\A^R-\tA^R\|_*.
\end{align}

In the two dimensional case, for $\Sigma,\rho>0$, we consider the
smaller class
\begin{align}
V_{\Sigma,\rho}:=\{(a,k)~~\mbox {as in}~
\eqref{bounded_coeff_2d}:~\|a\|_\infty\leq\Sigma,~\|k\|_{\infty}\leq\rho\}.
\end{align}

\begin{theorem}\label{main_thm_2d} For any
$\Sigma>0$, there exists $0<\rho\leq 1$ depending only on $\Sigma$,
$R$ and $c_R$, such that the following holds: If $(a,k),(\ta,\tk)\in
V_{\Sigma,\rho}$, then
\begin{equation}
\Delta(\langle a,k\rangle,\langle \ta,\tk\rangle)\leq
C\norm\A-\tA\norm_*,
\end{equation}
where $\Delta$ is with respect to $L^\infty\times  L^\infty$, and
$C$ is a constant depending only on $\Sigma$, $\rho$, $c_R$ and $R$.
\end{theorem}
The proofs of Theorem \ref{main_thm} and \ref{main_thm_2d} are
based on the analysis of the singularities of the Schwartz kernel
of $\A$ as in \cite{choulliStefanov99} and, respectively,
\cite{stefanovUhlmann03}. We also use an extended version of an
estimate in \cite{balJollivet08}, see Section~\ref{Preliminary
Estimates} below.

One can formulate and prove similar results in the case where the
velocity belongs to an open subspace of $\Rm^n$, i.e., the speed
can change, as in \cite{balJollivet08,choulliStefanov99}. We
restrict ourselves to the fixed speed case ($|\theta|=1$) for the
sake of simplicity of the exposition.

\section{Isometric transportation of the albedo operators from $\partial\Omega$ to $\partial
B_R$}\label{albedo_Extension}
Recall that $\Omega$ is strictly convex
with $\overline\Omega\subset B_R$, and $\A$ is the albedo operator
for the radiation transport in $\Omega$. The coefficients $(a,k)$
are extended by zero in $B_R\setminus\overline\Omega$ and let $\A_R$
be the albedo operator corresponding to the radiative transport in
$B_R\times S^{n-1}$, which takes functions on $\Gamma_-^R$ to
functions on $\Gamma_+^R$; see \eqref{Gamma_B_R}. Recall that
$\tau_\pm(x,\theta)$ is the travel time it takes a particle at $x\in
B_R$ to reach the boundary $\partial B_R$ while moving in the
direction of $\pm\theta$.

We consider next the set of ``projections" of $\Gamma_\pm$ onto
$\Gamma^R_\pm$ defined by
\begin{align}\label{bd_projection}
\tilde{\Gamma}^R_\pm:=\{(x\pm\tau_\pm(x,\theta)\theta,\theta:~~(x,\theta)\in\Gamma_\pm\}\subsetneq\Gamma_\pm^R
\end{align}and the transportation maps
\begin{align}\label{transportation_def_1}
&[\T f](x',\theta'):=f(x'-\tau_-(x',\theta')\theta',\theta'),\quad\forall(x',\theta')\in\Gamma_-\\
&[\tilde{\T}f](x+\tau_+(x,\theta)\theta,\theta):=f(x,\theta),\quad\forall(x,\theta)\in\Gamma_+.\label{transportation_def_2}
\end{align}
$\T$ takes maps defined on $\tilde{\Gamma}^R_-$ to maps on
$\Gamma_-$, whereas $\tilde{\T}$ takes maps defined on $\Gamma_+$ to
maps on $\tilde{\Gamma}^R_+$.

\begin{proposition}For any $1\leq p\leq\infty$, the maps
\begin{align}
\T &:L^p(\tilde{\Gamma}^R_-;d\xi)\to L^p({\Gamma}_-;d\xi),\\
\tilde{\T}&:L^p({\Gamma}_+;d\xi)\to L^p(\tilde{\Gamma}^R_+;d\xi)
\end{align}are isomorphisms.
\end{proposition}

\begin{proof}
From the definitions \eqref{transportation_def_1} and
\eqref{transportation_def_2} we have
\begin{align}
\text{ess sup}_{\tilde\Gamma^R_-}f=\text{ess sup}_{\Gamma_-}[\T
f]\quad \mbox{and}\quad\text{ess
sup}_{\tilde\Gamma^R_-}[\tilde{\T}f]=\text{ess sup}_{\Gamma_-}f.
\end{align}This proves the isometry for $p=\infty$. For $1\leq p<\infty$ we have the identities
\begin{align}
\int_{\tilde{\Gamma}^R_-}|f(x,\theta)|^pd\tilde\xi(x,\theta)=\int_{\Gamma_-}|[\T f](x',\theta)|^pd\xi(x',\theta),\\
\int_{\tilde{\Gamma}^R_+}|\tilde{\T}f(x,\theta)|^pd\tilde\xi(x,\theta)=\int_{\Gamma_+}|f(x',\theta)|^pd\xi(x',\theta)
\end{align}
obtained by the change of variables $x=x'-\tau_-(x',\theta)$ and
$x=x'+\tau_+(x',\theta)$ respectively.
\end{proof}
\begin{proposition}
Let $(a,k)$ be an admissible pair for the transport in $\Omega$
such that the forward problem is well-posed in $L^p(\Omega\times
S^{n-1})$, $1\leq p\leq \infty$, $n\geq 2$, and let $\A$ be the
corresponding albedo operator. Extend the coefficients by zero in
$B_R\setminus\overline\Omega$. Then the forward problem in
$L^p(B_R\times S^{n-1})$, $1\leq p\leq \infty$, is also well posed
and let $\A^R$ be the corresponding albedo operator. For any
$f\in L^p(\Gamma^R_-)$ we have
\begin{equation}\label{albedoes_connection}
{\A}^R[ f](x,\theta)= \left\{
\begin{array}{ll}
f(x-\tau_-(x,\theta)\theta,\theta),&{if}~(x,\theta)\in\Gamma^R_+\setminus \tilde{\Gamma}^R_+,\\
\tilde{\T}{\A} \T f(x,\theta),&{if}~(x,\theta)\in\tilde{\Gamma}^R_+.
\end{array}
\right.
\end{equation}
\end{proposition}
\begin{proof}
The relation \eqref{albedoes_connection} is a direct consequence of
the fact that in $B_R\setminus\overline\Omega$, where the
coefficients vanish, the solution $u(x,\theta)$ of the transport
equation is constant in $x$ along the lines in the direction of
$\theta$.
\end{proof}

\begin{proposition}
Let $(a,k)$ and $(\ta,\tk)$ be admissible pairs for the transport in
$\Omega$ such that the forward problem is well-posed in
$L^p(\Omega\times S^{n-1})$, $n\geq 2$, $1\leq p\leq \infty$, and
let $\A$ and $\tilde{\A}$ be the corresponding albedo
operators. Extend the coefficients by zero in
$B_R\setminus\overline\Omega$ and let $\A^R$, respectively ${\tA}^R$,
be their corresponding albedo operators for the transport through
$B_R$. Then \eqref{isometry} and \eqref{isometry*} hold.
\end{proposition}
\begin{proof}Following \eqref{albedoes_connection}
\begin{equation}\label{albedo_difference}
(\A^R-\tA^R)[f](x,\theta)= \left\{
\begin{array}{ll}
0,&{if}~(x,\theta)\in\Gamma^R_+\setminus \tilde{\Gamma}^R_+,\\
\tilde{\T}[\A-\tA] \T
f(x,\theta),&{if}~(x,\theta)\in\tilde{\Gamma}^R_+.
\end{array}
\right.
\end{equation}For $f\in L^p(\Gamma^R_-)$,
\begin{align*}
\|(\A^R-\tA^R)f\|_{L^p(\Gamma^R_+)}&=\|\tilde{\T}(\A-\tA)\T
f\|_{L^p(\tilde{\Gamma}^R_+)}=
\|(\A-\tA)\T f\|_{L^p({\Gamma}_+)}\\
&\leq\|\A-\tA\|\cdot\|\T f\|_{L^p(\Gamma_-)}=\|\A-\tA\|\cdot\|f\|_{L^p(\tilde{\Gamma}^R_-)}\\
&\leq\|\A-\tA\|\cdot\|f\|_{L^p({\Gamma}^R_-)}.
\end{align*}Hence $\|\A^R-\tA^R\|\leq\|\A-\tA\|$.

To prove the converse inequality, let $f_0$ be the projection of $f$
on $\tilde{\Gamma}^R_+$, i.e.
\begin{equation*}
f_0:= \left\{
\begin{array}{ll}
0,~~on~~\Gamma^R_+\setminus \tilde{\Gamma}^R_+,\\
f,~~on~~\tilde{\Gamma}^R_+.
\end{array}
\right.
\end{equation*}
Then
\begin{align*}
\|(\A-\tA)\T f\|_{L^p({\Gamma}_+)}&=\|\tilde{\T}(\A-\tA)\T
f\|_{L^p(\tilde{\Gamma}^R_+)}=
\|(\A^R-\tA^R)f\|_{L^p(\Gamma^R_+)}\\
&=\|(\A^R-\tA^R)f_0\|_{L^p(\Gamma^R_+)}\leq\|\A^R-\tA^R\|\cdot\|f_0\|_{L^p({\Gamma}^R_-)}\\
&=\|\A^R-\tA^R\|\cdot\|f|_{\tilde{\Gamma}^R_-}\|_{L^p(\tilde{\Gamma}^R_-)}
=\|\A^R-\tA^R\|\cdot\|\T f\|_{L^p({\Gamma}_-)}.
\end{align*}Since $\T$ is onto ${L^p({\Gamma}_-)}$, the inequality
above yields $\|\A-\tA\|\leq\|\A^R-\tA^R\|$.

Next, we prove the isometry in the $\|\cdot\|_*$-norm. Let
$\alpha,\talpha, \alpha^R,\talpha^R$ be the Schwartz kernels
associated with the albedo operators $\A,\tA,\A^R,\tA^R$, and let
$A, \tiA,A^R,\tiA^R$ and $\beta,\tbeta,\beta^R,\tbeta^R$
be the coefficients from the  corresponding singular decomposition
as in \eqref{2d_singular_decomposition}.

On the one hand, for $f\in C^\infty(\Gamma^R_-)$ arbitrary, and
$(x,\theta)\in\Gamma_+$ we have:
\begin{align*}
[\tilde\T(\tA-\A)\T
f]&(x+\tau_+(x,\theta)\theta,\theta)=[(\tA-\A)\T f](x,\theta)\\
&=\int_{\Gamma_-}[\talpha-\alpha](x,\theta,x',\theta')\T
f(x',\theta')d\xi(x',\theta')\\
&=\int_{\Gamma_-}[\talpha-\alpha](x,\theta,x',\theta')
f(x'-\tau_-(x',\theta'),\theta')d\xi(x',\theta')\\
&=\int_{\tilde{\Gamma}^R_-}[\tgamma-\gamma](x,\theta,x_R',\theta')
f(x'_R,\theta')d^R\xi(x'_R,\theta').
\end{align*}
In the last equality above we change variables
$x'_R=x'-\tau_-(x',\theta')\theta'$, and denote the distribution
$[\tgamma-\gamma](x,\theta,x'-\tau_-(x',\theta')\theta',\theta'):=[\talpha-\alpha](x,\theta,x',\theta')$.

On the other hand, from \eqref{albedo_difference}, we get
\begin{align*}
[\tilde\T(\tA-\A)\T
f]&(x+\tau_+(x,\theta)\theta,\theta)=[\tA^R-\A^R]f(x+\tau_+(x,\theta)\theta,\theta)\\
&=\int_{{\Gamma}^R_-}[\talpha^R
-\alpha^R](x+\tau_+(x,\theta)\theta,\theta,x'_R,\theta')f(x'_R,\theta')d\xi^R(x^R,\theta')\\
&=\int_{\tilde{\Gamma}^R_-}[\talpha^R-\alpha^R](x+\tau_+(x,\theta)\theta,\theta,x'_R,\theta')
f(x'_R,\theta')d\xi^R (x^R,\theta').
\end{align*}
Therefore, in the sense of distributions
\begin{align}\label{eq*1}
[\talpha^R
-\alpha^R](x+\tau_+(x,\theta)\theta,\theta,x'-\tau_-(x',\theta')\theta',\theta')=[\talpha
-\alpha](x,\theta,x',\theta')
\end{align}
Independently of the equality \eqref{eq*1} above, due to the fact
that $a=\ta=0$ in the $B_R\setminus\Omega$, by direct verification
in the formula \eqref{strentgh_of_singularity}, we get
\begin{align}\label{eq*2}
[\tiA^R-A^R](x'-\tau_-(x',\theta')\theta',\theta')=[\tiA-A](x',\theta'),\quad(x',\theta')\in\Gamma_-.
\end{align}
Now \eqref{eq*1} and \eqref{eq*2} applied to
\eqref{2d_singular_decomposition} yield for
$(x,\theta,x',\theta')\in\Gamma_+\times\Gamma_-$:
\begin{align}\label{eq*3}
[\tbeta^R-\beta^R](x+\tau_+(x,\theta)\theta,\theta,x'-\tau_-(x',\theta')\theta',\theta')
=[\tbeta-\beta](x,\theta,x',\theta').
\end{align}
The isometric relation \eqref{isometry*} now follows from
\eqref{eq*1} and \eqref{eq*3}.
\end{proof}

\section{The forward problem in three or higher dimensions}\label{preliminaries_3D}
In this section we recall the properties which define the albedo
operator and its kernel's singular expansion when $n\geq 3$.

Let ${T}$ be the operator defined by the left hand side of
\eqref{transport_eq} considered on $B_R\times S^{n-1}$, $n\geq 3$.
From \eqref{continuous_coeff}, the second and the third terms of
$T$ are bounded operators on $L^1(B_R\times S^{n-1})$, while the
first term is unbounded. We view ${T}$ as a (closed)  unbounded
operator on $L^1(B_R\times S^{n-1})$ with the domain
\[
D({T})= \{u\in L^1(B_R\times S^{n-1});\; \theta\cdot\nabla u\in
L^1(B_R\times S^{n-1}),\; u|_{\Gamma_-}=0\};
\]
see \cite{choulliStefanov99}.


We work under either one of the following {\em subcritical}
conditions that yield well-posedness for the boundary value
problem \eqref{transport_eq} and \eqref{bd_condition_in}:
\begin{equation}\label{critic_CS}
\text{ess sup}_{(x,\theta)\in\Omega\times S^{n-1}}
\left|\tau(x,\theta)\int_{S^{n-1}}k(x,\theta,\theta') d\theta'
\right|<1,
\end{equation}or
\begin{equation}\label{critic_DL}
a(x,\theta)-\int_{S^{n-1}}k(x,\theta,\theta') d\theta'\geq
0,~~a.e.~~\Omega\times S^{n-1};
\end{equation} see, e.g.,
\cite{balJollivet08,choulliStefanov99, dlen6, mokhtar, RS}.

Let $\delta_{\{x\}}(x')$ represent the delta distribution with
respect to the boundary measure $d\mu(x')$ supported at
$x\in\partial\Omega$, and let $\delta_{\{\theta\}}(\theta')$ represent
the delta distribution with respect to $d\theta$ centered at
$\theta\in S^{n-1}$. The following result is a recast of
\cite[Theorem 2.3]{choulliStefanov99} to the unit speed
velocities.

\begin{proposition}\label{singular_decomposition} Let $n\geq 3$.
Assume that the direct problem is well-posed. Then the albedo
operator $\A:L^1(\Gamma^R_-,d\xi)\to L^1(\Gamma^R_+,d\xi)$ is
bounded and its Schwartz kernel $\alpha(x,\theta,x',\theta')$,
considered as a distribution on $\Gamma_+$ with
$(x',\theta')\in\Gamma^R_-$ parameters, is given by
$\alpha=\alpha_1+\alpha_2+\alpha_3$, where
\begin{align}
\alpha_1(x,\theta,x',\theta')=&\frac{|n(x')\cdot\theta'|}{n(x)\cdot\theta}e^{-\int_0^{\tau_-(x,\theta)}a(x-t\theta,\theta)dt}
\delta_{\{x'+\tau_+(x',\theta')\theta'\}}(x)\delta_{\{\theta'\}}(\theta)\label{alpha_1}\\
\alpha_2(x,\theta,x',\theta')=&\frac{|n(x')\cdot\theta'|}{n(x)\cdot\theta}\int_0^{\tau_+(x',\theta')}
e^{-\int_0^{\tau_+(x'+t\theta',\theta)}a(x-s\theta,\theta)ds}
e^{-\int_0^t a(x'+s\theta',\theta')ds}\label{alpha_2}\\
&\times
k(x'+t\theta',\theta',\theta)\delta_{\{x'+t\theta'+\tau_+(x'+t\theta',\theta)\theta\}}(x)dt\nonumber\\
|n(x')\cdot \theta'|^{-1}\alpha_3\in&
L^\infty(\Gamma^R_-;L^1(\Gamma^R_+,d\xi)).\label{alpha_3}
\end{align}
\end{proposition}


\section{Preliminary estimates}\label{Preliminary Estimates}
This section extends a result from \cite{balJollivet08} for
continuous coefficients to the class of essentially bounded
coefficients as in \eqref{continuous_coeff}. This extension is
crucial to the transportation of the problem to a larger domain, if
no boundary knowledge of the coefficients is available, and has the added
benefit of simplifying the proof considerably.

Although the presentation below concerns the unit speed velocity
and measurements at the boundary, the results easily extend to the
variable velocity and measurements in the rotating planes setting.
The main novelty in this section is made possible by the following
result on the approximation of the identity.

\begin{proposition}\cite[Theorem 8.15]{folland}\label{app_id_1}
Suppose $\varphi\in L^1(\mathbb{R}^n)$ with $\int\varphi(x)dx=1$ and
$|\varphi(x)|\leq C(1+|x|)^{-n-\eps}$ for some $C,\eps>0$, and let
$\varphi_t=t^{-n}\varphi(x/t)$. If $f\in L^p(\mathbb{R}^n)$, for
some $1\leq p\leq \infty$, then $f*\varphi_t\to f(x)$ with $t\to
0^+$, for almost every $x$ in the Lebesgue set of $f$-- in
particular, for almost every $x\in \mathbb{R}^n$, and at every $x$
at which $f$ is continuous.
\end{proposition}

\begin{corollary} \label{app_id_cor}
There is a family of maps $\phi_{\eps,x_0',\theta_0'}\in
L^1(\Gamma_-,d\xi)$, $(x_0',\xi_0')\in\Gamma_-$ and $\eps>0$, such
that, for any $f\in L^\infty(\Gamma_+,d\xi)$ given,
\begin{align}\label{app_id_limit}
\lim_{\eps\to
0}\int_{\Gamma_-}\phi_{\eps,x_0',\theta_0'}(x',\theta')f(x',\theta')d\xi(x',\theta')=f(x_0',\theta_0'),
\end{align}whenever $(x_0',\theta_0')$ is in the Lebesgue set of
$f$. In particular, \eqref{app_id_limit} holds for almost every
$(x_0',\theta_0')\in\Gamma_-$.
\end{corollary}
\begin{proof}For $(x_0',\theta_0')\in \Gamma_-$ and $\eps>0$ sufficiently small, let
$(x',\theta'):U\times V\subset\mathbb{R}^{2n-2}\to
x'(U)\times\theta'(V)\subset\partial\Omega\times S^{n-1}$ be a
coordinate chart with $(x_0',\theta_0')\in x'(U)\times\theta'(U)$
such that $x_0'=x'(0)$ and $\theta_0'=\theta'(0)$. For
$(x',\theta')\in\Gamma_-$, define
\begin{align}\label{cutoff}
\phi_{\eps,x_0',\theta_0'}(x',\theta')=\frac{1}{|n(x')\cdot
\theta'|}
\left|\frac{D(u,v)}{D(x',\theta')}\right|\varphi_\eps(u(x'))\varphi_\eps(v(\theta')),
\end{align}where $\varphi(u)\equiv 1/(\omega_{n-1})$ for $|u|< 1$, $\varphi(u)\equiv 0$ for $|u|\geq 1$, and
$\varphi_\eps(u)=\eps^{-n+1}\varphi(u/\eps)$. By $\omega_{n-1}$ we
denoted the volume of the $(n-1)$-dimensional unit ball. Then, for
any $\eps>0$, $\int\varphi_\eps(u)du=1$ and
\begin{align*}
\int_{\Gamma_-}\phi_{\eps,x_0',\theta_0'}&(x',\theta')f(x',\theta')d\xi(x',\theta')=\int_{R^{2n-2}}\varphi_\eps(u)\varphi_\eps(v)f\left(x'(u),\theta'(v)\right)dudv.
\end{align*}Apply the equality above to $f\equiv 1$ to get
$\|\phi_{\eps,x_0',\theta_0'}\|_{L^1(\Gamma_+,d\xi)}=1$. The
conclusion follows from Proposition \ref{app_id_1}.
\end{proof}
We introduce the following notations. For $(x',\theta')\in\Gamma_-$,
$0\leq t\leq\tau_+(x',\theta')$ and $\theta\in S^{n-1}$ let
\begin{align}
x^{+}(t,x',\theta',\theta)&:=x'+t\theta'+\tau_+(x'+t\theta',\theta)\theta,\label{x+}\\
F(t,x',\theta',\theta)&:=e^{-\int_0^{\tau_+(x'+t\theta',\theta)}a(x'+t\theta'+s\theta,\theta)ds}e^{-\int_0^t
a(x'+s\theta',\theta')ds}\label{E},
\end{align}denote the exiting point after one scattering at the
point $x'+t\theta'$ coming from direction $\theta'$ into the
direction $\theta$, and, respectively, the total absorption along the
broken line due to this scattering. We note here that $x^+$ depends
on the point $x'+t\theta'$ where the scattering takes place and not
explicitly on $t$. This remark will be used later in
\eqref{almost_there}.

Let $(a,k),(\tilde{a},\tilde{k})$ be as in \eqref{continuous_coeff}
extended by 0 on $B_R\setminus\overline\Omega$, such that the
forward problem is well-posed. All the operators bearing the tilde
refer to $(\tilde{a},\tilde{k})$ and are defined in a similar way
with the ones for $(a,k)$. For example $\tilde{\mathcal{A}}$ is the
albedo operator corresponding to $(\tilde{a},\tilde{k})$. Recall
that $n$ is the dimension of the space.
\begin{theorem}\label{bal_lemma}Let $(a,k),(\tilde{a},\tilde{k})$ be as in
\eqref{continuous_coeff}. For almost every
$(x_0',\theta_0')\in\Gamma_-$ the following estimates hold:

For $n\geq 2$,
\begin{align}
\left|e^{-\int_0^{\tau_+(x_0',\theta_0')}a(x_0'+t\theta_0',\theta_0')ds}-
e^{-\int_0^{\tau_+(x_0',\theta_0')}\tilde{a}(x_0'+t\theta_0',\theta_0')ds}\right|
\leq\norm\mathcal{A}-\tilde{\mathcal{A}}\norm\label{bal_estimate1}.
\end{align}For $n\geq 3$,
\begin{align}
\int_0^{\tau_+(x_0',\theta_0')}\int_{S^{n-1}}|k-\tilde{k}|(x_0'+t\theta_0',\theta',\theta)E(t,x_0',\theta_0',\theta)d\theta
dt\leq\norm\mathcal{A}-\tilde{\mathcal{A}}\norm+\nonumber\\
+\|F-\tF\|_\infty\int_0^{\tau_+(x_0',\theta_0')}\int_{S^{n-1}}\tilde{k}(x_0+t\theta_0',\theta_0',\theta)d\theta
dt.\label{bal_estimate2}
\end{align}
\end{theorem}
\begin{proof}
Let $(x_0',\theta_0')\in\Gamma_-$ be arbitrarily fixed and let
$\phi_{\eps,x_0',\theta_0'}\in L^1(\Gamma_-)$ be defined as in
Corollary \ref{app_id_cor}. To simplify the formulas, since
$(x_0',\theta_0')$ is fixed, in the following we drop this
dependence from the notation $\phi_\eps=\phi_{\eps,x_0',\theta_0'}$.

Let $\mathcal{A}=\mathcal{A}_1+\mathcal{A}_2+\mathcal{A}_3$ be the
decomposition of the albedo operator given by
\[\mathcal{A}_i f(x,\theta)=\int_{\Gamma_-}\alpha_i(x,\theta,x',\theta')f(x',\theta')d\mu(x')d\theta',~~i=1,2,3,\]
where $\alpha_i$, $i=1,2,3$ are the Schwartz kernels in
Proposition \ref{singular_decomposition} and $d\mu(x')$ is the
induced Lebesgue measure on the boundary $\partial\Omega$.

Let $\phi\in L^\infty(\Gamma_+)$ with $\|\phi\|_\infty\leq 1$. Since
$\|\phi_\eps\|_{L^1(\Gamma_-)}=1$, the mapping properties of the
albedo operator imply that
\begin{align}\label{original_estimate}
\left|\int_{\Gamma_+}\phi(x,\theta)[\mathcal{A}-\tilde{\mathcal{A}}]\phi_\eps(x,\theta)d\xi(x,\theta)\right|\leq
\norm\mathcal{A}-\tilde{\mathcal{A}}\norm.
\end{align}

Next we evaluate each of the three terms in
$\int_{\Gamma_+}\phi(x,\theta)[\mathcal{A}-\tilde{\mathcal{A}}]\phi_\eps(x,\theta)d\xi(x,\theta)$
by using the decomposition in Proposition
\ref{singular_decomposition} and Fubini's theorem.

The first term is evaluated using the formula \eqref{alpha_1}:
\begin{align}
I_1(\phi,\eps):=\int_{\Gamma_+}\phi(x,\theta)[\mathcal{A}_1-\tilde{\mathcal{A}}_1]\phi_\eps(x,\theta)d\xi(x,\theta)=
\int_{\Gamma_-}\phi(x'+\tau_+(x',\theta'),\theta')\phi_\eps(x',\theta')\\
\times\left[e^{-\int_0^{\tau_+(x',\theta')}a(x'+s\theta',\theta')ds}-
e^{-\int_0^{\tau_+(x',\theta')}\tilde{a}(x'+s\theta',\theta')ds}\right]d\xi(x',\theta').\nonumber
\end{align}
Since the integrand above is in $L^\infty(\Gamma_-)$ by applying
\eqref{app_id_limit}, we get for almost every
$(x_0',\theta_0')\in\Gamma_-$
\begin{align}
I_1(\phi)(x_0',\theta_0')&:=\lim_{\eps\to
0}I_1(\phi,\eps)\nonumber\\
&=\phi(x_0'+\tau_+(x_0',\theta_0'),\theta_0')
\left(e^{-\int_0^{\tau_+(x_0',\theta_0')}a(x_0'+s\theta_0',\theta_0')ds}
-e^{-\int_0^{\tau_+(x_0',\theta_0')}\tilde{a}(x_0'+s\theta_0',\theta_0')ds}\right).\label{term1_lim}
\end{align}

 To evaluate the second term we use the notations \eqref{x+} and
 \eqref{E} and the formula \eqref{alpha_2}:
\begin{align*}
I_2(\phi,\eps)&:=\int_{\Gamma_+}\phi(x,\theta)[\mathcal{A}_2-\tilde{\mathcal{A}}_2]
\phi_\eps(x,\theta)d\xi(x,\theta)\nonumber\\
&=\int_{\Gamma_-}\phi_\eps(x',\theta')d\xi(x',\theta')
\Bigl\{\int_{S^{n-1}}\int_0^{\tau_+(x',\theta')}
\phi(x^{+}(t,x',\theta',\theta),\theta)\times\nonumber\\
&\times\left[F(t,x',\theta',x,\theta)k(x'+t\theta',\theta',\theta)-\tF(t,x',\theta',x,\theta)\tilde{k}(x'+t\theta',\theta',\theta)\right]dtd\theta\Bigr\}.
\end{align*}
Apply again \eqref{app_id_limit} for the continuous integrand
above to obtain for almost every $(x_0',\theta_0')\in\Gamma_-$:
\begin{align}
I_2(\phi)(x_0',\theta_0')&:=\lim_{\eps\to
0}I_2(\phi,\eps)\nonumber\\&=\int_{S^{n-1}}\int_0^{\tau_+(x_0',\theta_0')}
\phi(x^{+}(t,x_0',\theta_0',\theta),\theta)
\left[F(t,x_0',\theta_0',x,\theta)k(x_0'+t\theta_0',\theta_0',\theta)\right.\label{term2_lim}\\
&\left.-\tF(t,x_0',\theta_0',x,\theta)\tilde{k}(x_0'+t\theta_0',\theta_0',\theta)\right]dtd\theta,\nonumber
\end{align}or $I_2(\phi)=I_{2,1}(\phi)+I_{2,2}(\phi)$ with
\begin{align}
I_{2,1}(\phi)&=\int_{S^{n-1}}\int_0^{\tau_+(x_0',\theta_0')}
\phi(x^{+}(t,x_0',\theta_0',\theta),\theta)
F(t,x_0',\theta_0',x,\theta)(k-\tilde{k})(x_0'+t\theta_0',\theta_0',\theta)dtd\theta\label{term21_lim},\\
|I_{2,2}(\phi)|&\leq\int_{S^{n-1}}\int_0^{\tau_+(x_0',\theta_0')}
|F-\tF|(t,x_0',\theta_0',x,\theta)
\tilde{k}(x_0'+t\theta_0',\theta_0',\theta)dtd\theta\label{term22_lim}.
\end{align}

Consider the third term
\begin{align*}
I_3(\phi,\eps)&=\int_{\Gamma_+}\phi(x,\theta)[\mathcal{A}_3-\tilde{\mathcal{A}}_3]\phi_\eps(x,\theta)d\xi(x,\theta)\\
&=\int_{\Gamma_-}\phi_\eps(x',\theta')d\xi(x',\theta') \left\{
\int_{\Gamma_+}\phi(x,\theta)\frac{(\alpha_3-\tilde{\alpha}_3)(x,\theta,x',\theta')}{|n(x')\cdot\theta'|}
d\xi(x,\theta) \right\}.
\end{align*}
By \eqref{alpha_3}, the map
$(x',\theta')\mapsto\int_{\Gamma_+}\phi(x,\theta){(\alpha_3-\tilde{\alpha}_3)(x,\theta,x',\theta')}{|n(x')\cdot\theta'|}^{-1}
d\xi(x,\theta)$ is in $L^\infty(\Gamma_-)$, and then, by
\eqref{app_id_limit} we get for almost every
$(x_0',\theta_0')\in\Gamma_-$
\begin{align}
I_3(\phi)(x_0',\theta_0'):=\lim_{\eps\to 0}I_3(\phi,\eps)=
\int_{\Gamma_+}\phi(x,\theta)\frac{|\alpha_3-\tilde{\alpha}_3|(x,\theta,x_0',\theta_0')}{|n(x_0')\cdot\theta_0'|}
d\xi(x,\theta).\label{term3_lim}
\end{align}

The left hand side of \eqref{original_estimate} has three terms. We
move the third term to the right hand side (with absolute values)
and take the limit with $\eps\to 0$ to get
\begin{align}\label{original_estimate2}
\left|I_1(\phi)+I_2(\phi)\right|(x_0',\theta_0')\leq
\norm\mathcal{A}-\tilde{\mathcal{A}}\norm+
I_3(|\phi|)(x_0',\theta_0'),~a.e.~(x_0',\theta_0')\in\Gamma_-,
\end{align}for any $\phi\in L^\infty(\Gamma_+)$ with $\|\phi\|_\infty=1$.

We note that the negligible set on which the inequality above does
not hold may depend on $\phi$. In \eqref{original_estimate2}, we
shall choose two sequences of $\phi$ to conclude the estimates of
the lemma. Since countable union of negligible sets is negligible,
the inequality \eqref{original_estimate2} holds almost everywhere on
$\Gamma_-$, independently of the term in the sequence. This
justifies the argument below for almost every $(x_0',\theta_0')$ in
$\Gamma_-$.

First we show the estimate \eqref{bal_estimate1}. Let $\{\phi_m\}$
be a sequence of maps such that $|\phi_m|\le 1$, $\phi_m\equiv 1$
near $(x_0'+\tau_+(x_0',\theta_0')\theta_0',\theta_0')$ and with
support inside $\{(x,\theta)\in\Gamma_+:
~|x-x_0'+\tau_+(x_0',\theta_0')\theta_0'|+
|\theta-\theta_0'|<1/m\}$. Then \eqref{term1_lim} gives
\begin{align}
I_1(\phi_m)=e^{-\int_0^{\tau_+(x_0',\theta_0')}a(x_0'+s\theta_0',\theta_0')ds}-
e^{-\int_0^{\tau_+(x_0',\theta_0')}\tilde{a}(x_0'+s\theta_0',\theta_0')ds},
\end{align}independently of $m$. From \eqref{term2_lim} we have
$\lim_{m\to \infty}I_2(\phi_m)=0$ since the support in $\theta$
shrinks to $\theta_0'$. From \eqref{term3_lim} we also have
$\lim_{n\to \infty}I_3(|\phi_m|)=0$, since the support shrinks to
one point.

Next we prove the estimate \eqref{bal_estimate2}.
Let $\{\phi_{m,q}\}$ be a (double indexed) sequence defined by
$\phi_{m,q}(x,\theta)=\chi_m(\theta)\varphi_{q}(x,\theta)$ where
$\chi_m\in L^\infty(S^{n-1})$ with $\chi_m\equiv 0$ for
$|\theta-\theta_0'|\leq 1/m$ and $\chi_m\equiv 1$ for
$|\theta-\theta_0'|>1/m$ and $\varphi_q\in L^\infty(\Gamma_+)$ to
be specified below. Regardless of the way we define $\varphi_q$,
the presence of $\chi_m$ already yields $I_1(\phi_{m,q})=0$ for
all $m,q$ since
$\chi_{m}(x_0'+\tau_+(x_0',\theta_0')\theta_0',\theta_0')=0$.

Next we construct $\varphi_{q}$. Note that we only need to have it
defined for $\theta\neq\theta_0'$ since the function $\chi$
vanishes nearby $\theta_0'$.

Define first $\varphi_q(x,\theta)$ as follows. Using the notation
\eqref{x+}, let
\[\Pi_{x_0',\theta_0',\theta}:=\{x^+(t,x_0',\theta_0',\theta)\in\partial B_R: ~0\leq t\leq
\tau_+(x_0',\theta_0')\}\] represent the set of the exit points of
particles entering in the direction $\theta_0'$ at $x_0'$ which
scattered once in the direction of $\theta$. The set
$\Pi_{x_0',\theta_0',\theta}$ lies in the intersection of the
plane through $x_0'$ spanned by $\theta_0',\theta$ and the
boundary $\partial B_R$.

Now, for each $\theta\in S^{n-1}$ and
$x^+(t,x_0',\theta_0',\theta)$ of \eqref{x+} in the set
\begin{equation}\label{support}
\{x^+\in\partial B_R:~ \mbox{dist}_{\partial
B_R}(x^+,\Pi_{x_0',\theta_0',\theta})<1/m\},
\end{equation}
we define
\begin{align}\label{almost_there}
\varphi_q(x^+(t,x_0',\theta_0',\theta),\theta)
:=\mbox{sgn}(k-\tilde{k})(x_0'+t\theta_0',\theta_0',\theta).\end{align}
Outside the set in \eqref{support} we define
$\varphi_q(\cdot,\theta)\equiv 0$.

Since $x^+(t,x_0',\theta_0',\theta)$ in \eqref{x+} depends only on
$x_0'+t\theta_0',\theta_0'$ and $\theta$, the equation
\eqref{almost_there} gives a well defined function $\varphi_q\in
L^\infty(B_R\times S^{n-1})$. Note also that \eqref{support}
shrinks to a negligible set on $\partial B_R$ with $q\to \infty$.

Now apply the estimate \eqref{original_estimate2} to $\phi_{m,q}$
and use $I_1(\phi_{m,q})=0$ to get
\begin{align}
\left|I_{2,1}(\phi_{m,q})\right|(x_0',\theta_0')\leq
\norm\mathcal{A}-\tilde{\mathcal{A}}\norm+
I_3(|\phi_{m,q}|)(x_0',\theta_0')+|I_{2,2}(\phi_{m,q})|(x_0',\theta_0').
\end{align}Since the support of $\phi_{m,q}$ shrinks
to a set of measure zero in $\Gamma_+$ as $q\to\infty$, we get for
all $m$ and almost every $(x_0',\theta_0')\in\Gamma_-$,
$\lim_{q\to\infty}I_3(|\phi_{m,q}|)(x_0',\theta_0')=0$. Finally,
from \eqref{term21_lim} we obtain for almost every
$(x_0',\theta_0')\in\Gamma_-$ that
\begin{align*}
\lim_{m\to \infty}\lim_{n\to
\infty}I_{2,1}(\phi_{m,q})=\int_{S^{n-1}}\int_0^{\tau_+(x_0',\theta_0')}
F(t,x_0',\theta_0',x,\theta)|k-\tilde{k}|(x_0'+t\theta_0',\theta_0',\theta)dtd\theta,
\end{align*}while from \eqref{term22_lim} we have
\begin{align*}|I_{2,2}(\phi_{m,q})|\leq\int_{S^{n-1}}\int_0^{\tau_+(x_0',\theta_0')}
|F-\tF|(t,x_0',\theta_0',x,\theta)
\tilde{k}(x_0'+t\theta_0',\theta_0',\theta)dtd\theta,
\end{align*}for all $m,q$. The
estimate \eqref{bal_estimate2} in the lemma follows.
\end{proof}
\section{Stability modulo gauge transformations}
In this section we prove Theorem \ref{main_thm}.

We start with two pairs $(a,k),(\ta,\tk)\in U_{\Sigma,\rho}$ and let
$$\eps:=\norm\A-\tA\norm.$$ We shall find an intermediate pair
$(a',k')\sim (a,k)$ such that \eqref{closeness_in_a} and
\eqref{closeness_in_k} hold for some constant $C>0$ dependent on
$\Sigma,\rho$, $\Omega$ and $B_R$.

Define first the ``trial" gauge transformation:
\begin{align}\label{trial_gauge}
\varphi(x,\theta):=e^{-\int_0^{\tau_-(x,\theta)}(\ta-a)(x-s\theta,\theta)ds},~~a.e.~~(x,\theta)\in
B_R\times S^{n-1}.
\end{align}
Then $\varphi>0$,  $\varphi|_{\Gamma^R_-}=1$,
$\theta\cdot\nabla\varphi(x,\theta)\in L^\infty(B_R\times S^{n-1})$
and
\begin{align}\label{a1-a}
\ta(x,\theta)=a(x,\theta)-\theta\cdot\nabla_x\ln\varphi(x,\theta).
\end{align}Note that $\varphi|_{\Gamma^R_+}$ is close, but not equal,
to 1. We use the first estimate from Lemma~\ref{bal_lemma} to decide
how far from 1 can $\varphi|_{\Gamma^R_+}$ be.

By \eqref{bal_estimate1}, we have for almost every $(x_0',\theta_0')\in\Gamma_-$
\begin{align*}
\left|e^{-\int_0^{\tau_+(x_0',\theta_0')}a(x_0'+t\theta_0',\theta_0')ds}-
e^{-\int_0^{\tau_+(x_0',\theta_0')}\tilde{a}(x_0'+t\theta_0',\theta_0')ds}\right|
\leq\eps.
\end{align*}In each of the integrals of the left hand side above change variables
$t=\tau_+(x_0',\theta_0')-s$ and denote
$x_0=x_0'+\tau_+(x_0',\theta_0')\theta_0'$  to get
\begin{align}\label{trial_gauge_estimate1}
\left|e^{-\int_0^{\tau_-(x_0,\theta_0')}a(x_0-s\theta_0',\theta_0')ds}-
e^{-\int_0^{\tau_-(x_0,\theta_0')}\tilde{a}(x_0-s\theta_0',\theta_0')ds}\right|
\leq\eps.
\end{align}When $(x_0',\theta_0')$ covers
$\Gamma^R_-$ almost everywhere we get that $(x_0,\theta_0')$ covers $\Gamma^R_+$ almost everywhere.

Now apply the Mean Value theorem to $u\mapsto e^{-u}$ to get the
lower bound
\begin{align}\label{trial_gauge_estimate2}
\Bigl|e^{-\int_0^{\tau_-(x_0,\theta_0')}a(x_0-s\theta_0',\theta_0')ds}-&
e^{-\int_0^{\tau_-(x_0,\theta_0')}\tilde{a}(x_0-s\theta_0',\theta_0')ds}\Bigr|
\nonumber\\
&=
e^{-u_0}\left|\int_0^{\tau_-(x_0,\theta_0')}(\tilde{a}-a)(x_0-s\theta_0',\theta_0')ds\right|\nonumber\\
&= e^{-u_0}|\ln\varphi(x_0,\theta_0')|\geq
e^{-2R\Sigma}|\ln\varphi(x_0,\theta_0')|.
\end{align}
where $u_0=u_0(x_0,\theta_0',a,\tilde{a})$ is a value between the
two integrals appearing at the exponent in the left hand side
above, and $\varphi$ is defined in \eqref{trial_gauge}.

From \eqref{trial_gauge_estimate1} and \eqref{trial_gauge_estimate2}
we get the following estimate for the ``trial" gauge $\varphi$:
\begin{align}\label{trial_gauge_less_epsilon}
|\ln\varphi(x,\theta)|\leq
e^{2R\Sigma}\eps,~a.e.~(x,\theta)\in\Gamma^R_+.
\end{align}

The ``trial" gauge $\varphi$ is not good enough since it does not
equal 1 on $\Gamma^R_+$. We alter it to some $\tilde{\varphi}\in
L^\infty(B_R\times S^1)$ with
$\theta\cdot\nabla\ln\tilde{\varphi}\in L^\infty(B_R\times S^{n-1})
$ in such a way that $\tilde{\varphi}|_{\partial B_R}=1$. More
precisely, for almost every $(x,\theta)\in\overline{B}_R\times
S^{n-1}$, we define $\tilde{\varphi}(x,\theta)$ by
\begin{align}\label{true_gauge}
\ln\tilde{\varphi}(x,\theta):=\ln\varphi(x,\theta)-
\frac{\tau_-(x,\theta)}{\tau(x,\theta)}\ln\varphi(x+\tau_+(x,\theta)\theta,\theta).
\end{align}
Since $0\leq\tau_-(x,\theta)/\tau(x,\theta)\leq 1$ we get
$\tilde{\varphi}\in L^\infty(B_R\times S^1)$. Following directly
from its definition $\ln\tilde{\varphi}|_{\partial{B_R}}=0$: Indeed,
for $(x,\theta)\in\Gamma^R_-$, we get that $\tau_-(x,\theta)=0$ and
$\varphi(x,\theta)=1$, whereas for $(x,\theta)\in\Gamma^R_+$ we have
that $\tau_-(x,\theta)=\tau(x,\theta)$ and
$x=x+\tau_+(x,\theta)\theta$. Since both maps
$x\mapsto\tau(x,\theta)$ and
 $x\mapsto\ln\varphi(x+\tau_+(x,\theta)\theta,\theta)$ are constant
 in the direction of $\theta$ and since
 $\theta\cdot\nabla_x\tau_-(x,\theta)=1$, we get
 $\theta\cdot\nabla\ln\tilde{\varphi}(x,\theta)\in L^\infty(B_R\times
 S^{n-1})$ and
\begin{align}\label{closeness_varphi}
\theta\cdot\nabla\ln\tilde{\varphi}(x,\theta)=\theta\cdot\nabla\ln{\varphi}(x,\theta)-
\frac{\ln\varphi(x+\tau_+(x,\theta)\theta,\theta)}{\tau(x,\theta)}.
\end{align}

Define now the pair $(a',k')$ in the equivalence class of $\langle
a,k\rangle$ by
\begin{align}
a'(x,\theta)&:=a(x,\theta)-\theta\cdot\nabla_x\ln\tilde{\varphi}(x,\theta),\label{tildea-a}\\
k'(x,\theta',\theta)&:=\frac{\tilde{\varphi}(x,\theta)}{\tilde{\varphi}(x',\theta')}k(x,\theta',\theta).\label{tildek-k}
\end{align}
Now $\A'$, the albedo operator corresponding to $(a',k')$,
satisfies $\A'=\A$, and then
\begin{align}
\norm\A'-\tA\norm=\norm\A-\tA\norm=\eps.
\end{align}

Next we compare the pairs $(a',k')$ with $(\ta,\tk)$ and show them
to satisfy \eqref{closeness_in_a} and \eqref{closeness_in_k}.

Using the definitions \eqref{a1-a}, \eqref{tildea-a}, the relation
\eqref{closeness_varphi}, and the estimate for $\varphi$ on
$\Gamma^R_+$ \eqref{trial_gauge_less_epsilon}, we have for almost
every $(x,\theta)\in B_R\times S^{n-1}$:
\begin{align}\label{a1-tildea}
|\ta(x,\theta)-a'(x,\theta)|&= |[\ta-a](x,\theta)\nonumber
+[a-a'](x,\theta)|\\
&=|\theta\cdot\nabla_x\ln\tilde{\varphi}(x,\theta)-\theta\cdot\nabla_x\ln{\varphi}(x,\theta)|\\
&=\frac{|\ln\varphi(x+\tau_+(x,\theta)\theta,\theta)|}{\tau(x,\theta)}\leq
\eps\frac{ e^{2R\Sigma}}{\tau(x,\theta)}.\nonumber
\end{align}

Since the coefficients are supported away from $\partial B_R$ (by
construction of $B_R$) such that \eqref{nontangential} holds,
following \eqref{a1-tildea} we obtain the estimate
\eqref{closeness_in_a} in the form
\begin{align}\label{final_estimate_for_a}
\|\ta-a'\|_\infty\leq \eps\frac{ e^{2R\Sigma}}{c_R},
\end{align}with $c_R$ from \eqref{nontangential}.

Up to this point, all the arguments above also work for two
dimensional domains.

Next we prove the estimate \eqref{closeness_in_k}. These arguments
are specific to three or higher dimensions. Recall the formula
\eqref{E} adapted to $a'$
\begin{align*}
F'(t,x',\theta',\theta)=e^{-\int_0^{\tau_+(x'+t\theta',\theta)}a'(x'+t\theta'+s\theta,\theta)ds}e^{-\int_0^t
a'(x'+s\theta',\theta')ds}
\end{align*}and note that $F=F'$ is a quantity preserved under the
gauge transformation \eqref{tildea-a}. This follows by direct
calculation and the fact that $\tilde{\varphi}=1$ on $\partial B_R$.
Then for almost all $(x',\theta')\in\Gamma^R_-$,
$t\in[0,\tau(x',\theta')]$ and $\theta\in S^{n-1}$, we have the
following lower bound
\begin{align}\label{lower_bound_E}
|F'(t,x',\theta',\theta)|\geq e^{-4R\Sigma}.
\end{align}Using the non-negativity of $\ta$ and $a'$ we estimate
\begin{align}
|[\tF-F']&(t,x',\theta',\theta)|\leq \left| e^{-\int_0^t
\ta(x'+s\theta',\theta')ds}-e^{\int_0^t
a'(x'+s\theta',\theta')ds}\right|\nonumber\\
&+\left|e^{-\int_0^{\tau_+(x'+t\theta',\theta)}\ta(x'+t\theta'+s\theta,\theta)ds}
-e^{-\int_0^{\tau_+(x'+t\theta',\theta)}a'(x'+t\theta'+s\theta,\theta)ds}\right|\nonumber\\
&\leq\left|\int_0^t [\ta-a'](x'+s\theta',\theta')ds\right|+
\left|\int_0^{\tau_+(x'+t\theta',\theta)}[\ta-a'](x'+t\theta'+s\theta,\theta)ds\right|\nonumber\\
&\leq\eps
e^{2R\Sigma}\left(\int_0^t\frac{ds}{\tau(x'+s\theta',\theta')}+
\int_0^{\tau_+(x'+t\theta',\theta)}\frac{ds}{\tau(x'+t\theta'+s\theta,\theta)}\right)\nonumber\\
&=\eps e^{2R\Sigma}\left(\frac{t}{\tau(x',\theta')}+
\frac{\tau_+(x'+t\theta',\theta)}{\tau(x'+t\theta',\theta)}\right)\leq2\eps
e^{2R\Sigma}.\label{E-E}
\end{align}The next to the last inequality uses \eqref{a1-tildea};
the following equality uses the fact that both maps
$s\mapsto\tau(x'+s\theta',\theta')$ and $s\mapsto
\tau(x'+t\theta'+s\theta,\theta)$ are constant in $s$, while the
last inequality uses the fact that $t\leq \tau(x',\theta')$ and
$\tau_+(x'+t\theta'+s\theta,\theta)\leq\tau(x'+t\theta'+s\theta,\theta)$.
Therefore we proved that for almost all $(x',\theta')\in\Gamma_-$,
$t\in[0,\tau(x',\theta')]$ and $\theta\in S^{n-1}$ we have
\begin{align}\label{E1-tildeE}
|[\tF-F'](t,x',\theta',\theta)|\leq 2\eps e^{2R\Sigma}.
\end{align}
Recall now the estimate \eqref{bal_estimate2} with respect to the
pairs $(a',k')$ and $(\ta,\tk)$:
\begin{align}
\int_0^{\tau_+(x_0',\theta_0')}\int_{S^{n-1}}|\tk-k'|(x_0'+t\theta_0',\theta',\theta)F'(t,x_0',\theta_0',\theta)d\theta
dt\leq\norm\tA-\A'\norm+\nonumber\\
+\|\tF-F'\|_\infty\int_0^{\tau_+(x_0',\theta_0')}\int_{S^{n-1}}\tk(x_0+t\theta_0',\theta_0',\theta)d\theta
dt.
\end{align}Now use the lower bound for $F'$ in \eqref{lower_bound_E}, the upper bound for
$\|\tF-F'\|_\infty$  in \eqref{E1-tildeE} and the hypothesis
$\|\tk\|_{\infty,1}\leq\rho$, to obtain
\begin{align}
\int_0^{\tau_+(x_0',\theta_0')}\int_{S^{n-1}}|\tk-k'|(x_0'+t\theta_0',\theta',\theta)d\theta
dt\leq\eps e^{4R\Sigma}\left(1+2\rho e^{2R\Sigma}\right).
\end{align}
Finally, integrating the formula above in
$(x_0',\theta_0')\in\Gamma_-$ with the measure
$d\xi(x_0',\theta_0')$, we get
\begin{align}
\|\tk-k'\|_1\leq \eps \pi Re^{4R\Sigma}\left(1+2\rho
e^{2R\Sigma}\right).
\end{align}Now choose
\begin{align*}
C=\max\{\pi Re^{4R\Sigma}\left(1+2\rho
e^{2R\Sigma}\right),{e^{2R\Sigma}}/{c_R}\}
\end{align*}with $c_R$ from \eqref{nontangential} to finish the proof of Theorem \ref{main_thm}.

\section{Preliminaries for two dimensional domains}\label{preliminaries_2D}
This section introduces the framework for the problem in two
dimensions. The results are mainly from \cite{stefanovUhlmann03}.

As above, ${T}$ denotes the operator defined by the left hand side
of \eqref{transport_eq} in $B_R\times S^1$ with $B_R\subset\Rm^2$.
The coefficients are extended to be 0 in $B_R\setminus \Omega$. By
the regularity assumption \eqref{bounded_coeff_2d}, the second and
the third terms of $T$ are bounded operators in
$L^\infty(B_R\times S^1)$. The first term is unbounded. We view
${T}$ as a (closed) unbounded operator on $L^\infty(B_R\times
S^1)$ with the domain
\[
D({T})= \{u\in L^\infty(B_R\times S^1);\; \theta\cdot\nabla u\in
L^\infty(B_R\times S^1),~~ u|_{\Gamma^R_-}\in
L^\infty(\Gamma^R_-)\}.
\]
To simplify notation, for $x\neq y$, we denote by
\[\widehat{x-y}=\arg(x-y)=\frac{x-y}{|x-y|},
\]the direction from $y$ to $x$. Also, let
\begin{align}
\label{att_y->x} 0<
E(y,x)=e^{-\int_0^1a(x-t(x-y);\widehat{x-y})dt}\leq 1
\end{align}denote the attenuation along the segment in the direction from $y$ to
$x$. The attenuations corresponding to $a'$ and $\ta$ will be
denoted by $E'$ and $\tE$, respectively.

 The boundary value problem
\eqref{transport_eq} and \eqref{bd_condition_in} is equivalent to
the operator equation
\begin{align}\label{integral_eq}
(I-M)u=Jf_-,
\end{align}
where, using \eqref{att_y->x},
\begin{align}
Jf_-(x,\theta)&=E(x-\tau_-(x,\theta)\theta,x)f_-(x-\tau_-(x,\theta)\theta,\theta)\label{J}\\
Kf(x,\theta)&=\int_{S^1}k(x,\theta',\theta)f(x,\theta')d\theta'\label{K},\quad\mbox{and}\\
Mf(x,\theta)&=\int_0^\infty
E(x-t\theta,x)Kf(x-t\theta,\theta)dt.\label{M}
\end{align}
Under the subcritical assumption
\begin{align}\label{subcritical_2D}
R\|k\|_\infty<1/2,
\end{align}
the operator $M:L^\infty(B_R\times S^1)\to L^\infty(B_R\times
S^1)$ is contractive and \eqref{integral_eq} has a unique solution
obtained by Neumann series. Moreover, for $f_-\in
L^\infty(\Gamma^R_-)$, we get $u=(I-M)^{-1}Jf_-\in D(T)$ has a
well defined trace in $L^\infty(\Gamma^R_+)$ given by
\begin{align}\label{trace}
\gamma[u](x,\theta):=
[(I-M)^{-1}Jf_-]|_{\Gamma^R_+}(x,\theta),~~(x,\theta)\in\Gamma^R_+;
\end{align}
see \cite{choulliStefanov99,stefanovUhlmann03}. Therefore the
albedo operator $\A:L^\infty(\Gamma^R_-)\to L^\infty(\Gamma^R_+)$
is bounded and has the Schwartz kernel
$\alpha(x,\theta,x',\theta')=\phi|_{\Gamma^R_+}(x,\theta;x',\theta')$,
where, for $(x',\theta')\in\Gamma^R_-$, the map
$(x,\theta)\mapsto\phi(x,\theta;x',\theta')$ is the fundamental
solution of \eqref{transport_eq} subject to the boundary condition
\begin{align}\label{phi-}
\phi|_{\Gamma^R_-}(\cdot,\cdot;x',\theta')=|n(x')\cdot\theta'|^{-1}\delta_{\{x'\}}(\cdot)\delta_{\{\theta'\}}(\cdot).
\end{align}
More precisely, as shown in \cite[Proposition
1]{stefanovUhlmann03}, $\alpha=\alpha_0+\alpha_1+\alpha_2$ with
\begin{align*}
\alpha_0+\alpha_1+\alpha_2=\gamma\phi_0+
\gamma\phi_1+\gamma\phi_2:=\phi_0|_{\Gamma^R_+}+
M\phi_0|_{\Gamma^R_+}+(I-M)^{-1}M^2\phi_0|_{\Gamma^R_+},
\end{align*}where
\begin{align}\phi_0&=E(x-\tau_-(x,\theta)\theta,x)\delta_{\{\theta'\}}(\theta)
\int_0^{\tau_+(x',\theta')}\delta(x-x'-t\theta')dt,\label{phi0}\\
\phi_1&=\frac{\chi(y)k(y,\theta',\theta)}{|\theta'\times\theta|}E(y-\tau_-(y,\theta')\theta',y)E(y,y+\tau_+(y,\theta)\theta)
\label{phi1}
,\\
0\leq\phi_2&\leq
d_R\|k\|_\infty^2\left(1-\ln|\theta'\times\theta|\right).\label{phi2}
\end{align}The constant $d_R$ depends on $R$ only, $\chi$ is the
characteristic function of $B_R$, and, for
$(x,\theta,x',\theta')\in B_R\times S^1\times\Gamma^R_-$,
$y=y(x,\theta,x',\theta')$ is the point of intersection of the
rays $x'\to x'+\infty\theta'$ and $x\to x-\infty\theta$. The
kernel $\beta$ in \eqref{2d_singular_decomposition} is then given
by
\begin{align}\label{beta}
\beta(x,\theta,x',\theta')=[\gamma\phi_1+
\gamma\phi_2](x,\theta,x',\theta'),~(x,\theta,x',\theta')\in\Gamma^R_+\times\Gamma^R_-,
\end{align}where $\gamma$ is the trace operator on $\Gamma^R_+$.

The following estimate from \cite{stefanovUhlmann03} is needed
later.
\begin{lemma}\label{SU_lemma}Let $\chi$ be the characteristic function of
$B_R$, $L(x',\theta')$ be the line through $x'$ in the direction
$\theta'$ and $dl(y)$ be the Lebesgue measure on the line. Then
\begin{align}
\int_0^\infty\chi(x-t\theta)\int_{L(x',\theta')}\frac{\chi(y)}{|x-t\theta-y|}dl(y)dt\leq
C(1-\ln|\theta'\times\theta|),
\end{align}where $C$ is a constant dependent on $R$ only.
\end{lemma}

\section{Stability of the equivalence classes in two dimensions}
In this section we work under the hypotheses in Section
\ref{preliminaries_2D} and prove Theorem \ref{main_thm_2d}.

Let $(a,k),(\ta,\tk)\in V_{\Sigma,\rho}$ be given with
$\|\A-\tA\|_*=\eps$. Define the pair $(a',k')$ in the equivalence
class of $\langle a,k\rangle$ by \eqref{tildea-a} and
\eqref{tildek-k} as before, i.e.
\begin{align*}
a'(x,\theta):=a(x,\theta)-\theta\cdot\nabla_x\ln\tilde{\varphi}(x,\theta),\quad
k'(x,\theta',\theta):=\frac{\tilde{\varphi}(x,\theta)}{\tilde{\varphi}(x',\theta')}k(x,\theta',\theta),
\end{align*}where $\tilde{\varphi}$ is given in
\eqref{closeness_varphi}.

Then the corresponding albedo operator $\A'=\A$ and, thus,
\begin{align}\|\A'-\tA\|_*=\|\A-\tA\|_*=\eps.\end{align}In particular,
\begin{align}\label{closeness_beta}
\|\tilde{\beta}-\beta'\|_\infty\leq\eps,
\end{align}and
\eqref{trial_gauge_estimate1} holds.

Starting with \eqref{trial_gauge_estimate1}, the same arguments as
the ones in the three dimensional domains, notably
\eqref{a1-tildea}, are valid in two dimensions to conclude the
estimate \eqref{final_estimate_for_a}:
\begin{align}\label{final_estimate_for_a_2D}
\|\ta-a'\|_\infty\leq \eps (e^{2R\Sigma}/{c_R})=:\eps\tilde{C}.
\end{align}In turn, \eqref{final_estimate_for_a_2D}
yields
\begin{align}\label{quotient_varphi}
\frac{\tilde{\varphi}(x,\theta)}{\tilde{\varphi}(x',\theta')}=
e^{-\int_0^{\tau_-(x,\theta)}(a'-\ta)(x-s\theta,\theta)ds+\int_0^{\tau_-(x',\theta')}(a'-\ta)(x-s\theta',\theta')ds}\leq
e^{4R\tilde{C}\eps}.
\end{align}

From the definition \eqref{tildek-k} and \eqref{quotient_varphi} we
now get
\begin{align}\label{rho'}
\|k'\|_\infty\leq\rho e^{4R\tilde{C}\eps}.
\end{align}

Let
\begin{align}\label{broken_path_att}
\tEE(y,\theta',\theta):=\tE(y-\tau_-(y,\theta')\theta')\tE(y,y+\tau_+(y,\theta)\theta).
\end{align}be the total attenuation along the broken path due to one scattering at
$y\in B_R$, when coming from the direction $\theta'$ and
scattering into the direction $\theta$. The formula \eqref{phi1}
now reads
\begin{align}
\tilde\phi_1(x,\theta,x',\theta')=\frac{\chi(y)[\tk\tEE](y,\theta',\theta)}{|\theta'\times\theta|},
~(x,\theta,x',\theta')\in B_R\times S^1\times\Gamma_-,
\label{phi1_bis}
\end{align}where $\chi(y)$ and $y=y(x,\theta,x',\theta')$
are as described above in Section \ref{preliminaries_2D}. We also
consider $E'_1$ and $\phi_1'$ defined similarly with the
attenuation $a'$ to replace $\ta$.

The relation with the quantity in \eqref{E},
$F'(t,x',\theta',\theta)=E'_1(x'+t\theta',\theta',\theta),$ allows
us to use the estimates \eqref{lower_bound_E} and \eqref{E-E} to
conclude
\begin{align}|E'_1(y,\theta',\theta)|&\geq e^{-4R\Sigma}, ~~(y,\theta',\theta)\in
B_R\times S^1\times S^1,\label{estimate_on_F}\\
\|\tEE-E'_1\|_\infty&\leq 2\eps
e^{2R\Sigma},\label{estimate_on_F-F}.
\end{align}

Now use \eqref{estimate_on_F}, \eqref{estimate_on_F-F},
\eqref{phi1_bis} and \eqref{beta} in $E'_1(\tk-k')=
(E'_1-\tEE)\tk+(\tEE\tk-E'_1k'),$ (evaluated at
$(y,\theta',\theta)$ with $y=y(x,\theta,x',\theta')$ as above,) to
estimate
\begin{align}
e^{-4R\Sigma}|\tk-k'|&\leq |E'_1-\tEE|\tk+ |\tEE\tk-E'_1k'|\nonumber\\
&\leq |E'_1-\tEE|\tk+
|\tilde\beta-\beta'|~|\theta\times\theta'|+|\gamma\phi_2'-\gamma\tilde{\phi}_2||\theta\times\theta'|\nonumber\\
&\leq 2\eps\rho
e^{2R\Sigma}+\eps+|\gamma\phi_2'-\gamma\tilde{\phi}_2||\theta\times\theta'|\label{foundation}.
\end{align}
We estimate the last term using the identity
\begin{align}
|\theta'\times\theta||\gamma\phi_2'-\gamma\tilde{\phi}_2|&=|\theta'\times\theta|\gamma(I-M')^{-1}M'^2\phi_0'
-|\theta'\times\theta|\gamma(I-\tilde{M})^{-1}\tilde{M}^2\tilde{\phi}_0\nonumber\\
&=|\theta'\times\theta|\gamma(I-M')^{-1}[M'^2\phi_0'-\tilde{M}^2\tilde{\phi}^2_0]\label{estimate_alpha2}\\
&+|\theta'\times\theta|\gamma(I-\tilde{M})^{-1}[M'-\tilde{M}](I-M')^{-1}\tilde{M}^2\tilde{\phi}_0\nonumber
\end{align}
To estimate the first term in the right hand side above we write
\[[M'^2\phi_0'-\tilde{M}^2\tilde{\phi}^2_0]=M'(M'-\tilde{M})\phi_0'+(M'-\tilde{M})\tilde{M}\phi_0'+\tilde{M}^2(\phi_0'-\tilde\phi_0)\]
and bound each of the terms as follows. From their definitions we
have
\begin{align*}
\tilde{M}M'&\phi_0'=\int_0^\infty\tE(x-t\theta,x)dt\times
\\
&\times\int_{L(x',\theta')}\frac{\tk(x-t\theta,\widehat{x-t\theta-y},\theta)
k'(y,\theta',\widehat{x-t\theta-y})}{|x-t\theta-y|}E'(x',y)E'(y,x-t\theta)dl(y).
\end{align*}
\begin{align*}
M'M'&\phi_0'=\int_0^\infty E'(x-t\theta,x)dt\times\nonumber
\\
&\times\int_{L(x',\theta')}\frac{k'(x-t\theta,\widehat{x-t\theta-y},\theta)
k'(y,\theta',\widehat{x-t\theta-y})}{|x-t\theta-y|}E'(x',y)E'(y,x-t\theta)dl(y).
\end{align*}
Since $(I-M')^{-1}$ is bounded in $L^\infty$, with a norm
dependent on the radius only, say $C(R)$, by adding and
subtracting one term and by using Lemma \ref{SU_lemma}, we
estimate

\begin{align}
|\gamma(I-M')^{-1}(\tilde{M}-M')M'\phi_0'|\leq&
C(R)\|\tE-E'\|_\infty\|k\|_\infty\|k'\|_\infty
(1-\ln|\theta'\times\theta)\nonumber\\
&+C(R)\|\tk-k'\|_\infty\|k'\|_\infty(1-\ln|\theta'\times\theta|).\label{prelim_I1}
\end{align}
Now, from \eqref{final_estimate_for_a_2D} we get
\begin{align}
\|\tE-E'\|_\infty\leq 2R\|\ta-a'\|\leq C(R,\Sigma,c_R)\eps,
\end{align}for some constant which only depends on $R,\Sigma,c_R$.

In what follows we keep the notation $C(R,\Sigma,c_R)$ for constants
that may be different from equation to equation but they only depend
on $R,\Sigma, c_R$ in an explicit, but inessential, way.

Using the fact that $0\leq t(1-\ln t)\leq 1$ for $t\in [0,1]$, the
bound $\|\tk\|\leq\rho$ and the bound in \eqref{rho'}, we get from
\eqref{prelim_I1} that
\begin{align}\label{I11}
|\theta'\times\theta||\gamma(I-M')^{-1}(\tilde{M}-M')M'\phi_0'|\leq
C(R,\Sigma,c_R)(\eps\rho^2+ \rho\|\tk-k'\|_\infty).
\end{align}

By reversing the roles of $M'$ and $\tilde{M}$, we get similarly
\begin{align}\label{I12}
|\theta'\times\theta||\gamma(I-M')^{-1}(M'-\tilde{M})\tilde{M}\phi_0'|\leq
C(R,\Sigma,c_R)(\eps\rho^2+ \rho\|\tk-k'\|_\infty).
\end{align}
Similarly, from the definition of $\tilde{M}^2$ as above, we also
get
\begin{align}\label{I13}
|\theta'\times\theta||\gamma(I-M')^{-1}\tilde{M}^2(\phi_0'-\tilde\phi_0)\leq
C(R,\Sigma,c_R)\eps\rho^2.
\end{align}

The estimates \eqref{I11}, \eqref{I12} and \eqref{I13} imply
\begin{align}\label{I1}
|\theta'\times\theta|\gamma(I-M')^{-1}[M'^2\phi_0'-\tilde{M}^2\tilde{\phi}^2_0]\leq
C(R,\Sigma,c_R)(\eps\rho^2+ \rho\|\tk-k'\|_\infty).
\end{align}

Next we estimate the second term of the right hand side of
\eqref{estimate_alpha2}. Since for any $f\in L^\infty(B_R\times
S^1)$, we have
\begin{align*}
|[M'-\tilde{M}]f(x,\theta)|&\leq\left|\int_0^\infty[E'-\tE](x-t\theta,x)\int_{S^1}k'(x-t\theta,\theta',\theta)f(x-t\theta)d\theta'dt\right|\\
&+\left|\int_0^\infty\tE(x-t\theta,x)\int_{S^1}[k'-\tk](x-t\theta,\theta',\theta)f(x-t\theta)d\theta'dt\right|\\
&\leq\left\{2R\|E'-\tE\|_\infty\|k'\|_\infty+2R\|k'-\tk\|_\infty\right\}\|f\|_\infty\\
&\leq
\left\{C(R,\Sigma,c_R)\eps\rho+2R\|k'-\tk\|_\infty\right\}\|f\|_\infty
\end{align*}we get
\begin{align}\label{I2}
|\theta'\times\theta|~|\gamma(I-\tilde{M})^{-1}(M'-\tilde{M})^{-1}\tilde{M}^2\tilde{\phi}_0|\leq
C(R,\Sigma,c_R)\left\{\eps \rho^3+\|k'-\tk\|_\infty\rho^2\right\}.
\end{align}
Since $\rho\leq 1$, by applying \eqref{I1} and \eqref{I2} in
\eqref{estimate_alpha2}, we get
\begin{align}
|\gamma[\phi_2']-\gamma[\tilde{\phi}_2]| ~|\theta\times\theta'|\leq
C(R,\Sigma,c_R)\left\{\eps\rho^2+ \rho\|\tk-k'\|_\infty\right\}.
\end{align}
Therefore the basic estimate \eqref{foundation} yields
\begin{align}
\|k'-\tk\|_\infty\leq C(R,\Sigma,c_R)\eps+
C(R,\Sigma,c_R)\rho\|k'-\tk\|_\infty.
\end{align}
By choosing
\begin{align}
\rho<\frac{1}{C(R,\Sigma,c_R)},
\end{align}we get the final estimate
\begin{align}\label{final_estimate_k}
\|k'-\tk\|_\infty\leq \frac{C(R,\Sigma,c_R)}{1-C(R,\Sigma,c_R)}\eps.
\end{align}
The constant $C$ from Theorem \ref{main_thm_2d} is the largest
between the constant in \eqref{final_estimate_for_a_2D} and
\eqref{final_estimate_k}.
\section{Concluding Remarks}
In the case of an anisotropic attenuating medium, the albedo
operator determines the attenuation and scattering properties up
to a gauge equivalence class. The set of gauge functions has a
natural structure of a multiplicative group which acts
transitively on the pairs of the coefficients.

We showed that the gauge equivalent classes are stably determined
by the albedo operator. We understand the distance between
equivalent classes to be the infimum of the distances between the
corresponding representatives.

The proof uses essentially the fact that, without loss of
generality, the problem can be transferred to a larger domain, and,
consequently, the total travel time (with respect to the larger
domain) of free moving particles in the interior domain stays away
from zero. The no loss of generality part is due to the extension of
an estimate in \cite{balJollivet08} to essentially bounded
coefficients.

The fact that we get Lipschitz  stability estimates in
\eqref{closeness_in_a}, \eqref{closeness_in_k} instead of
conditional H\"older stability estimates as in
\cite{balJollivet08} may seem strange. In fact, if we  assume that
$a$ and $\tilde a$ depend on $x$ only (or on $(x,|\theta|)$, if
$\theta$ belongs to an open velocity space), then
\eqref{closeness_in_a} implies
\[
\int(a-\tilde a)(x+t\theta) \, dt=O(\eps)
\]
in the $L^\infty$ norm, compare with
\cite[Theorem~3.2]{balJollivet08}. Then, by using interpolation
estimates, and the stability of the X-ray transform, we can get a
conditional H\"older stability estimate for $a-\tilde a$, similar to
the one in \cite[Theorem~3.4]{balJollivet08}.

\section*{Acknowledgment} The third author thanks Jason Swanson
for pointing out the approximation of the identity result in
\cite[Theorem 8.15]{folland}.

\end{document}